\pdfoutput=1
\documentclass{shinyart}

\usepackage[utf8]{inputenc}
\usepackage{enumitem} 
\usepackage{comment}

\usepackage{shinybib}
\usepackage{biblatex}
\usepackage{enumitem}
\usepackage{siunitx}

\newtheorem{assumption}{Assumption}
\newtheorem{myalg}{Algorithm}
\crefname{figure}{Figure}{Figures}

\newcommand{\dx}{\mathop{}\!\mathrm{d}x}

\newcommand{\dt}{\mathop{}\!\mathrm{d}t}

\renewcommand{\phi}{\varphi}

\def \R{\mathbb{R}}

\def \N{\mathbb{N}}
\def \1{\mathbb{1}}
\def \eps{\varepsilon}
\def \Uad{{U_{\mathrm{ad}}}}

\newcommand{\meas}{\operatorname{meas}}

\usepackage[svgnames]{xcolor}
\usepackage{graphicx}

\usepackage{booktabs}
\usepackage{pgfplots}
\usepgfplotslibrary{colorbrewer}
\pgfplotsset{cycle list/Dark2-3}
\pgfplotsset{every axis legend/.append style={font=\footnotesize}}

\title{Error estimates for the approximation of multibang control problems}

\author{Christian Clason\thanks{%
    Faculty of Mathematics, University Duisburg-Essen, 45117 Essen, Germany (\email{christian.clason@uni-due.de}, \email{tram.do@uni-due.de})}
    \and Thi Bich Tram Do\footnotemark[1]
    \and Frank P\"orner\thanks{%
    Department of Mathematics, University of W\"urzburg, Emil-Fischer-Str. 40, 97074 W\"urzburg, Germany (\email{frank.poerner@mathematik.uni-wuerzburg.de})}\footnotemark[2]
}

\begin{document}

\maketitle

\begin{abstract}
    This work is concerned with optimal control problems where the objective functional consists of a tracking-type functional and an additional ``multibang'' regularization functional that promotes optimal control taking values from a given discrete set pointwise almost everywhere.
    Under a regularity condition on the set where these discrete values are attained, error estimates for the Moreau--Yosida approximation (which allows its solution by a semismooth Newton method) and the discretization of the problem are derived. Numerical results support the theoretical findings.
\end{abstract}

\section{Introduction}

We consider linear-quadratic optimal control problems where the optimal control is only allowed to take values at discrete values $u_1 < \dots < u_d \in \R$ with $d \in \N$. Such problems occur, e.g., in topology optimization, nondestructive testing or medical imaging; a similar task also arises as a sub-step in segmentation or labeling problems in image processing. However, such problems are inherently nonconvex and, more importantly, not weakly lower semi-continuous and hence cannot be treated by standard techniques. A classical remedy is convex relaxation, where the nonconvex constraint $u(x) \in \{u_1,\dots,u_d\}$ is replaced by the convex constraint $u(x)\in[u_1,u_d]$, but this leads to ignoring the intermediate parameter values.
In \cite{CK:2013,CK:2015,ClasonDo:2017,CKK:2017,CTW:2016}, it was therefore proposed to promote all desired control values using a convex \emph{multibang} penalty
\begin{equation*}
    G(u): L^2(\Omega) \to \R, \qquad u \mapsto \int\limits_\Omega g(u(x)) \dx,
\end{equation*}
for a suitable convex integrand $g: \R \to \R$ with a polyhedral epigraph whose vertices correspond to the desired control values $u_1,\dots,u_d$.
We thus consider the \emph{multibang control problem}
\begin{equation}\label{eq:intro:unreg}
    \min\limits_{u \in L^2(\Omega )} \frac{1}{2}\|Ku - z\|_Y^2 + \alpha G(u)
\end{equation}
with $\alpha > 0$, $z \in Y$ for a Hilbert space $Y$, and $K:L^2(\Omega)\to Y$ a linear and continuous operator (e.g., the solution operator for a linear elliptic partial differential equation).
Just as in $L^1$ regularization for sparsity (and in linear optimization), it can be expected that minimizers are found at the vertices of $G$, thus yielding the desired structure. 
Furthermore, it was shown in \cite{CIK:2014,CK:2015,ClasonDo:2017} that this leads to a primal-dual optimality system that can be solved by a superlinearly convergent semismooth Newton method in function space \cite{Ulbrich:2011,Kunisch:2008a} if a suitable Moreau--Yosida approximation (of the Fenchel conjugate $G^*$, see \cref{thm:optimality_set_reg} below) is introduced. It turns out that this approximation can be expressed in primal form as 
\begin{equation}\label{eq:intro:reg}
    \min\limits_{u \in L^2(\Omega )} \frac{1}{2}\|Ku - z\|_Y^2 + \alpha G(u) + \frac{\gamma}{2} \|u\|^2_{L^2(\Omega)}
\end{equation}
for a parameter $\gamma>0$.
We remark that this approach (i.e., applying the approximation to $G^*$ instead of $G$) does not destroy the non-differentiability of $G$ and hence preserves the structural properties of \eqref{eq:intro:unreg}.
Standard lower semicontinuity techniques can then be applied to show that the solutions to \eqref{eq:intro:reg} converge weakly to the solution to \eqref{eq:intro:unreg} as $\gamma\to 0$; see \cite[\S\,4.1]{CK:2015}. The aim of this paper is to establish strong convergence and in particular approximation error estimates for $\|\bar u - u_\gamma\|_{L^2(\Omega)}$. 

Let us recall some literature and already known results.
For the case $d=2$ we obtain the minimization problem
\begin{equation}\label{eq:intro:bangbang}
    \min\limits_{u_1 \leq u \leq u_2} \frac{1}{2}\|Ku - z\|_Y^2.
\end{equation}
and if the associated adjoint state $\bar p(x) \neq 0$ almost everywhere, it is well-known that $\bar u$ exhibits a bang-bang structure, i.e. $\bar u(x) \in \{u_1, u_2\}$ almost everywhere. This problem has been studied intensively in the literature, see \cite{Stadler2009,Troeltzsch2010,Wachsmuth2013,Wachsmuth:2011b,Wachsmuth:2011a} and the references therein. Note that this list is far away from being complete. For this problem a structural assumption has been established in \cite{Wachsmuth:2011b,Wachsmuth:2011a}, which controls the behavior of the adjoint state around a singular set and guarantees that the optimal control $\bar u$ exhibits a bang-bang structure. Using this assumption, error estimates for the approximation of \eqref{eq:intro:bangbang} can be proven; see \cite{Wachsmuth:2011b}. A related question is the Moreau--Yosida approximation of state constraints; see \cite{Himue:2009,Schiela:2014}.

If $d=3$ and $u_1 < u_2 = 0 < u_3$, the problem \eqref{eq:intro:unreg} resembles the minimization problem
\begin{equation}\label{eq:intro:bangbangoff}
    \min\limits_{u_1 \leq u \leq u_2} \frac{1}{2}\|Ku - z\|_Y^2 + \alpha \|u\|_{L^1(\Omega)},
\end{equation}
see, e.g., \cite{Stadler2009}.
The structural assumption used to prove error rates for the approximation of \eqref{eq:intro:bangbang} can be generalized to problem \eqref{eq:intro:bangbangoff}. Again, approximation error estimates can be proven; see \cite{Wachsmuth2013,Wachsmuth:2011b,Wachsmuth:2011a} and the reference therein. 

We will generalize this structural assumption to the multibang control problem \eqref{eq:intro:unreg}. We will show that this assumption is sufficient to guarantee that an optimal control $\bar u$ of \eqref{eq:intro:unreg} satisfies $\bar u(x) \in \{u_1,\dots,u_d\}$ for almost all $x \in \Omega$. Furthermore, we will use this condition to prove approximation error estimates of the form
\begin{equation*}
    \|\bar u - u_\gamma\|_{L^2(\Omega)} = \mathcal{O} \left( \gamma^{\frac{\kappa}{2}} \right)
\end{equation*}
with a constant $\kappa>0$ depending only on the structural assumption.

The paper is organized as follows. In \cref{sec:prelim} we recall some preliminary results which are needed for the convergence analysis. Our structural assumption is introduced in \cref{sec:reg_estimates} and used to derive the approximation error estimates. This is also the main result of this paper. In \cref{sec:disc_error}, we establish discretization error estimates under our structural assumption. We introduce an active set method for the solution of \eqref{eq:intro:reg} and show its equivalence to a semismooth Newton method in \cref{sec:active_set}. Finally, numerical results to support our theoretical findings can be found in \cref{sec:numerics}.

\section{Preliminary results}\label{sec:prelim}

Let $u_1 < u_2 < \dots < u_d$ be some given real numbers with $d \geq 2$, and let $\Omega \subset \R^n$ be a bounded domain. Following \cite{CK:2013,CK:2015,CKK:2017,ClasonDo:2017}, we define the piecewise linear function
\begin{equation*}
    g(v) := \begin{cases}
        \frac{1}{2}( (u_i + u_{i+1})v - u_i u_{i+1} ) &\text{if } v \in [u_i, u_{i+1}], \quad 1 \leq i < d,\\
        \infty & \text{else}.
    \end{cases}
\end{equation*}
As the pointwise supremum of affine functions, $g$ is convex and continuous on the interior of its domain $\mathrm{dom}(g) = [u_1,u_d]$. Hence, the corresponding integral functional
\begin{equation*}
    G: L^2(\Omega) \to \R, \qquad u \mapsto \int\limits_\Omega g(u(x)) \dx,
\end{equation*}
is proper, convex and weakly lower semicontinuous as well; see, e.g., \cite[Proposition 2.53]{Barbu}.

We now consider the problem
\begin{equation}\label{eq:prob_noUad}
    \min\limits_{u \in L^2(\Omega)} \frac{1}{2}\|Ku - z\|_Y^2 + \alpha G(u)
\end{equation}
with a parameter $\alpha > 0$. Standard semi-continuity methods then yield existence of a minimizer $\bar u$, which is unique if $K$ is injective; see~\cite{CK:2015}. 
We will later impose a condition which guarantees that $\bar u$ exhibits a multibang structure, i.e., $\bar u(x) \in \{u_1,\dots,u_d\}$ for almost every $x\in \Omega$.

Let us further define the set
\begin{equation*}
    \Uad := \{u \in L^2(\Omega): u_1 \leq u(x) \leq u_d \} = \mathrm{co} \left\{ u \in L^2(\Omega): u(x) \in \{u_1, \dots ,u_d\} \right\},
\end{equation*}
where $\mathrm{co}$ denotes the convex hull.
It is clear that \eqref{eq:prob_noUad} is equivalent to the problem
\begin{equation}\label{eq:main}\tag{$P$}
    \min\limits_{u \in \Uad} \frac{1}{2}\|Ku - z\|_Y^2 + \alpha G(u).
\end{equation}
We will use this equivalent formulation to derive variational inequalities which will be useful in the convergence analysis.
Standard convex analysis techniques then yield primal--dual optimality conditions; see, e.g.,\cite{CK:2015,ClasonDo:2017}.
\begin{proposition}\label{thm:optimality_set}
    Define the sets
    \begin{align*}
        Q_1 &:= \left\{q: q < \frac{\alpha}{2}(u_1 + u_2)\right\},\\
        Q_i &:= \left\{q: \frac{\alpha}{2}(u_{i-1} + u_i) < q < \frac{\alpha}{2}(u_i + u_{i+1}) \right\}, \quad 1 < i < d,\\
        Q_d &:= \left\{ q: q > \frac{\alpha}{2}(u_{d-1} + u_d) \right\},\\
        Q_{i,i+1} &:= \left\{ q: q = \frac{\alpha}{2}(u_i + u_{i+1}) \right\}.
    \end{align*}
    Let $\bar u \in \Uad$ with associated adjoint state $\bar p := K^\ast(z - K \bar u)$. Then $\bar u$ is a solution to \eqref{eq:main} if and only if
    \begin{equation}\label{eq:oc}
        \bar u(x) \in 
        \begin{cases}
            \{u_i\} \quad &\text{if } \bar p(x) \in Q_i \quad 1 \leq i \leq d,\\
            [u_i, u_{i+1}] \quad & \text{if } \bar p(x) \in Q_{i,i+1} \quad 1 \leq i < d.
        \end{cases}
    \end{equation}
\end{proposition}
It is clear that the optimal solution $\bar u$ is uniquely determined by the adjoint state on the sets $\{x\in\Omega: \bar p(x)\in Q_i\}$. 
We see furthermore that $\bar u(x) \in \{ u_1, \dots, u_d \}$ almost everywhere on $\Omega$ if $\meas \{x\in \Omega:\bar p(x)\in Q_{i,i+1}\} = 0$ for all $1 \leq i < d$. Hence $\bar u$ has a multibang structure in this case. In the following, we will make use of this relation to construct a suitable regularity condition on these sets. 
\begin{remark}\label{rem:alpha}
        Although the dependence of the optimal controls on $\alpha$ is not the focus of this work -- see instead the earlier works \cite{CK:2013,CK:2015,CTW:2016,CKK:2017}, and, in particular, \cite[Section 5]{ClasonDo:2017} -- let us recall the essential features for the sake of completeness. First, note that $\alpha$ enters the optimality conditions \eqref{eq:oc} only via the case distinction for the sets $Q_i$ and $Q_{i,i+1}$. Specifically, increasing the value of $\alpha$ shifts the conditions on $\bar p$ so that desired control values  $u_i$ of smaller magnitude are preferred. Conversely, for $\alpha\to0$, these conditions coincide with the well-known optimality conditions for bang-bang control problems where only $Q_1$, $Q_d$, and $Q_{1,d}$ are relevant; see, e.g., \cite[Lemma 2.26]{Troeltzsch2010}.
        This implies that apart from singular cases where $\meas\{x\in \Omega:\bar p(x) =c \}\neq 0$ for some $c\in\R$, the value of $\alpha$ does not influence the `` strength'' of the multibang penalty in enforcing the desired control values but only the specific selection among these values.
\end{remark}

We next introduce the  Moreau--Yosida approximation of \eqref{eq:main} with a regularization parameter $\gamma > 0$,
\begin{equation}\label{eq:main_tik}\tag{$P_\gamma$}
    \min\limits_{u \in \Uad} \frac{1}{2}\|Ku - z\|_Y^2 + \alpha G(u) + \frac{\gamma}{2} \|u\|_{L^2(\Omega)}^2.
\end{equation}
As for \eqref{eq:main}, arguments from convex analysis lead to the following optimality conditions; see \cite{CK:2015,ClasonDo:2017}.
\begin{proposition}\label{thm:optimality_set_reg}
    Define the sets
    \begin{align*}
        Q_1^\gamma &:= \left\{q: q < \frac{\alpha}{2} \left( \left(1+2 \frac{\gamma}{\alpha} \right)u_1 + u_2 \right) \right\},\\
        Q_i^\gamma &:= \left\{q: \frac{\alpha}{2}\left( u_{i-1} + \left(1 + 2 \frac{\gamma}{\alpha} \right)u_i \right) < q < \frac{\alpha}{2} \left( \left(1+2 \frac{\gamma}{\alpha} \right)u_i + u_{i+1} \right) \right\} ,\\
        Q_{i,i+1}^\gamma &:= \left\{q: \frac{\alpha}{2}\left( \left( 1+ 2 \frac{\gamma}{\alpha} \right)u_i + u_{i+1} \right) \leq q \leq \frac{\alpha}{2} \left( u_i + \left( 1 + 2 \frac{\gamma}{\alpha} \right)u_{i+1} \right) \right\},\\
        Q_d^\gamma &:= \left\{ q: \frac{\alpha}{2} \left( u_{d-1} + \left(1+2 \frac{\gamma}{\alpha} \right)u_d \right) < q \right\}.
    \end{align*}
    Let $u_\gamma \in \Uad$ with associated adjoint state $p_\gamma := K^\ast(z - K u_\gamma)$. Then $u_\gamma$ is a solution to \eqref{eq:main_tik} if and only if
    \begin{equation}\label{eq:opt_reg}
        u_\gamma(x) =
        \begin{cases}
            u_i \quad &\text{if } p_\gamma(x) \in Q_i^\gamma \quad 1 \leq i \leq d,\\
            \frac1\gamma \left(p_\gamma(x) -\tfrac\alpha2(u_i+u_{i+1})\right)\quad & \text{if } p_\gamma(x) \in Q_{i,i+1} \quad 1 \leq i < d.
        \end{cases}
    \end{equation}
\end{proposition}
We remark that \eqref{eq:opt_reg} is the explicit pointwise characterization of $u_\gamma \in (\partial G^*)_\gamma(p_\gamma)$, where $(\partial G^*)_\gamma$ denotes the Yosida approximation of the convex subdifferential (which coincides with the Fréchet derivative of the Moreau envelope) of the Fenchel conjugate of $G$, which justifies the term \emph{Moreau--Yosida approximation}; see, e.g., \cite[\S\,4.1]{ClasonDo:2017}.

We can also derive purely primal first-order optimality conditions for \eqref{eq:main} and \eqref{eq:main_tik} in terms of variational inequalities using standard arguments as in, e.g., \cite[Thm.~2.22]{Troeltzsch2010}.
\begin{proposition}\label{thm:opt_variational}
    Let $\bar u$ and $u_\gamma$ be solutions of \eqref{eq:main} and \eqref{eq:main_tik} with associated adjoint states $\bar p := K^\ast(z - K\bar u)$ and $p_\gamma := K^\ast(z - K u_\gamma)$, respectively. Then,
    \begin{align*}
        \left( -\bar p, u - \bar u \right)_{L^2(\Omega)} + \alpha G'(\bar u; u - \bar u) &\geq 0 \quad \text{for all } u \in \Uad,\\
        \left( -p_\gamma + \gamma u_\gamma, u - u_\gamma \right)_{L^2(\Omega)} + \alpha G'(u_\gamma; u - u_\gamma) &\geq 0 \quad \text{for all } u \in \Uad.
    \end{align*}
\end{proposition}
Here, $G'(\bar u; u - \bar u)$ denotes the directional derivative of $G$ at $\bar u$ in direction $u - \bar u$, which will be characterized in the following lemma. Note that for $\bar u, u \in \Uad$ we have $u- \bar u \in T_\Uad(\bar u)$ for 
\[
    T_\Uad(u) := \left\{ v\in L^2(\Omega): v(x) \left.  \begin{cases} \geq 0 & \text{if }  u(x) = u_1\\ \leq 0 & \text{if } u(x) = u_d \end{cases} \right\} \right\},
\]
i.e., the tangential cone to $\Uad$ in the point $u$.
It thus suffices to consider directional derivatives for directions in $T_{\Uad}$, which helps to avoid unnecessary case distinctions in the proof. Furthermore, since $\Uad \subset L^\infty(\Omega)$, we only have to consider directions in $L^\infty(\Omega)$. In the following, all pointwise expressions and calculations are understood in an almost everywhere sense.
\begin{lemma}\label{lem:G_dir_derivative}
    Let $u \in \Uad$ and define the sets
    \begin{align*}
        S_i &:= \{x \in \Omega: u(x) = u_i\}, \quad i =1,\dots,d,\\
        T_i &:= \{ x \in \Omega: u_i < u(x) < u_{i+1} \}, \quad i=1,\dots,d-1.
    \end{align*}
    The directional derivative of $G$ in direction $v \in T_\Uad(u) \cap L^\infty(\Omega)$ is then given as 
    \begin{equation*}
        \begin{aligned}
            G'(u;v) &= \sum\limits_{i=1}^{d-1} \int\limits_{T_i} \frac{1}{2}(u_i + u_{i+1})v(x) \ \dx\\
                    &\quad + \sum\limits_{i=1}^{d} \left[\ \int\limits_{S_i \cap \{v \geq 0\}} \frac{1}{2}(u_i + u_{i+1})v(x) \dx + \int\limits_{ S_i \cap \{ v < 0 \} } \frac{1}{2}(u_{i-1} + u_i) v(x) \dx \right].
        \end{aligned}
    \end{equation*}
\end{lemma}
\begin{proof}
    We use the definition of the directional derivative and of the sets $S_i$ and $T_i$ to obtain
    \begin{multline*}
        G'(u;v) := \lim\limits_{\rho \to 0} \frac{1}{\rho} \left( G(u+ \rho v) - G(u) \right)\\
        = \lim\limits_{\rho \to 0} \frac{1}{\rho} \left[ \sum\limits_{i=1}^{d-1} \int\limits_{T_i} ( g(u(x)+\rho v(x)) - g(u(x)) ) \dx + \sum\limits_{i=1}^d \int\limits_{S_i}( g(u(x)+\rho v(x)) - g(u(x)) ) \dx \right] .
    \end{multline*}
     We now make use of our assumption that $v \in T_{\Uad} \cap L^\infty(\Omega)$. For such a $v$, we can find a $\rho > 0$ such that $u + \rho v \in \Uad$. Note that this is a pointwise condition, which we are going to exploit in the following. We have to differentiate between several cases.
    \begin{enumerate}[label=(\roman*)]
        \item First, assume that $x \in T_i$ with $1 \leq i \leq d-1$. For $\rho$ small enough we then get $u(x) + \rho v(x) \in [u_i, u_{i+1}]$. Hence we obtain
            \begin{equation}\label{eq:G_dir1}
                \begin{aligned}[b]
                    g(u(x)+\rho v(x)) - g(u(x)) &= \frac{1}{2}( (u_i + u_{i+1})( u(x)+ \rho v(x)) - u_i u_{i+1} )\\
                    \MoveEqLeft[-1]- \frac{1}{2}( (u_i + u_{i+1})u(x) - u_i u_{i+1} )\\
                    &= \frac{\rho}{2} (u_i + u_{i+1}) v(x).
                \end{aligned}
            \end{equation}
            which yields
            \begin{equation*}
                \lim\limits_{\rho \to 0} \int\limits_{T_i} (g(u(x)+\rho v(x)) - g(u(x))) \dx = \int\limits_{T_i} \frac{1}{2}(u_i + u_{i+1}) v(x) \dx.
            \end{equation*}
        \item Now assume that $x \in S_i$ with $1 < i < d$. Then by definition, $u(x) = u_i$. Here we have to further differentiate between three cases.
            \begin{enumerate}[align=left]
                \item[$v(x)=0$:] Here we obtain $u(x) + \rho v(x) = u(x)$, leading to
                    \begin{equation*}
                        g(u(x) + \rho v(x)) - g(u(x)) = 0.
                    \end{equation*}
                \item[$v(x)>0$:] Here we obtain $u(x) + \rho v(x) \in [u_i, u_{i+1}]$ for $\rho$ small enough, leading to
                    \begin{equation*}
                        g(u(x)+\rho v(x)) - g(u(x)) = 
                        \frac{\rho}{2} (u_i + u_{i+1}) v(x).
                    \end{equation*}
                \item[$v(x)<0$:] Here we obtain $u(x) + \rho v(x) \in [u_{i-1}, u_{i}]$, leading as in \eqref{eq:G_dir1} to
                    \begin{equation*}
                        g(u(x)+\rho v(x)) - g(u(x)) = \frac{\rho}{2} ( u_{i-1} + u_i ) v(x).
                    \end{equation*}
            \end{enumerate}
            Combining all three cases yields 
            \begin{multline*}
                \lim\limits_{\rho \to 0} \frac{1}{\rho} \int\limits_{S_i} (g(u(x)+\rho v(x)) - g(u(x))) \dx \\
                = \int\limits_{ S_i \cap \{v \geq 0\} } \frac{1}{2}(u_i +u_{i+1} ) v(x) \dx + \int\limits_{ S_i \cap \{v < 0\} } \frac{1}{2} (u_{i-1} + u_i) v(x) \dx.
            \end{multline*}
        \item
            We are left with the special cases $x\in S_i$ for $i=1$ and $i = d$. We only consider the case $i=1$ as the case $i=d$ is similar. Hence we assume $x \in S_1$, which implies $u(x) = u_1$. Since $v \in T_\Uad(u)$, we have that $v(x) \geq 0$. If $v(x) > 0$, we obtain for $\rho$ small enough that $u(x) + \rho v(x) \in T_1$ holds, leading to
            \begin{equation*}
                g(u(x) + \rho v(x)) - g(u(x)) = \frac{\rho}{2}(u_1 + u_2) v(x)
            \end{equation*}
            and similar if $v(x) = 0$. This leads to
            \begin{equation*}
                \lim\limits_{\rho \to 0} \frac{1}{\rho} \int\limits_{S_1} ( g(u(x)+\rho v(x)) - g(u(x)) ) \dx= \int\limits_{S_1} \frac{1}{2}(u_1 + u_2) v(x) \dx.
            \end{equation*}
            A similar argument for the remaining case $i=d$ finishes the proof.
            \qedhere
    \end{enumerate}
\end{proof}

\section{Regularity assumption and error estimates}\label{sec:reg_estimates}

We now extend the active set condition from \cite{Wachsmuth:2011a,Wachsmuth:2011b} to the multibang control problem. From \cref{thm:optimality_set}, we see that the optimal control $\bar u$ is not uniquely determined by the adjoint state $\bar p$ on the \emph{singular sets} $Q_{i,i+1}$. We therefore need to control the way in which $\bar p$ ``detaches'' from these sets. This motivates the following assumption.
\renewcommand{\theassumption}{REG}
\begin{assumption}\label{reg}
    For the solution $\bar u$ to \eqref{eq:main} with adjoint state $\bar p = K^\ast(z-K \bar u)$ there exists a constant $c > 0$ and $\kappa > 0$ such that
    \begin{equation*}
        \meas\left( \bigcup\limits_{i=1}^{d-1} \left\{ x \in \Omega: \left|\bar p(x) - \frac{\alpha}{2}(u_i + u_{i+1}) \right| < \eps \right\} \right) \leq c \eps^\kappa
    \end{equation*}
    holds for all $\eps > 0$ small enough.
\end{assumption}
Note that if $\bar u$ satisfies this assumption, the sets $Q_{i,i+1}$ have Lebesgue measure zero. Hence, $\bar u$ is multibang by \cref{thm:optimality_set}. In addition, we have the following result, which is a direct consequence of $\meas \{x\in \Omega:\bar p(x)\in Q_{i,i+1}\} = 0$. 
\begin{lemma}\label{lemma:multi_bang}
    Assume $\bar u$ satisfies \cref{reg}. Then $\bar p(x) \in Q_i$ if and only if $\bar u(x) = u_i$ holds almost everywhere in $\Omega$.
\end{lemma}

Following \cite[Lemma 1.3]{Hinze:2012}, we can derive a sufficient condition for \cref{reg}.
\begin{theorem}\label{thm:suffreg}
    Suppose that the adjoint state $\bar p\in C^1(\bar \Omega)$ and satisfies 
    \[
        \min\limits_{x \in K_i} |\nabla p(x)| > 0\qquad\text{for all } i= 1,\dots,d-1,
    \]
    where
    \[
        K_i := \left\{ x \in \bar \Omega: p(x) = \frac{\alpha}{2}(u_i + u_{i+1}) \right\}.
    \]
    Then \cref{reg} holds with $\kappa = 1$.
\end{theorem}

\begin{proof}
    Define for $t\in\R$ the level sets $F_t := \{ x \in \bar \Omega: p(x) = t \}$. Now we use a continuity argument to obtain constants $\eps_0, c_0, C > 0$ such that for all $|t - \frac{\alpha}{2}(u_i + u_{i+1})|\leq \eps_0$ and all $1\leq i < d$ there holds
    \[
        |\nabla p(x)| \geq c_0 >0, \quad \mathcal{H}^{n-1}(F_t) \leq C,
    \]
    where $\mathcal{H}^{n-1}$ is the $(n-1)$-dimensional Hausdorff measure. 
    In the following, we denote by $\1_C$ the characteristic function of the set $C$, i.e., $\1_C(x) =1$ if $x\in C$ and $0$ else. 
    We now use the co-area formula
    \[
        \int\limits_\Omega h(x) |\nabla p(x)| \dx = \int\limits_{- \infty}^\infty \left(\ \int\limits_{p^{-1}(t)} h(x) \mathrm{d} {\mathcal{H}^{n-1}}(x) \right) \dt
    \]
    with the function
    \[
        h(x) := \1_{ E_i }, \quad E_i := \left\{ x \in \Omega: \left|p(x) - \frac{\alpha}{2}(u_i + u_{i+1}) \right| \leq \eps \right\} ,
    \]
    to obtain for all $1\leq i < d$ and $0 < \eps \leq \eps_0$ that
    \[
        c_0 \meas \left( E_i \right) \leq \int\limits_{E_i} |\nabla p(x)| \dx = \int\limits_{-\eps}^\eps \mathcal{H}^{n-1} \left( F_{t - \frac{\alpha}{2}(u_i + u_{i+1})}\right) \dt \leq 2 C \eps
    \]
    holds. Since this holds for all $1 \leq i < d$, the \cref{reg} now follows with $\kappa = 1$.
\end{proof}

\bigskip

We now establish error estimates for the approximation \eqref{eq:main_tik} of \eqref{eq:main}.
For this purpose, we first derive a stronger version of \cref{thm:opt_variational}. The next result, which is similar to ones in \cite{PoernerWachsmuth2016,PoernerWachsmuth2018}, is the most important tool in the convergence analysis.
\begin{lemma}\label{lem:improved_first_order}
    Assume that the solution $\bar u$ to \eqref{eq:main} satisfies \cref{reg}. Then,
    \begin{equation*}
        (-\bar p,u - \bar u)_{L^2(\Omega)} + \alpha G'(\bar u; u - \bar u) \geq c_A \|u-\bar u\|^{1 + \frac{1}{\kappa}}_{L^1(\Omega)} \qquad \forall u \in \Uad
    \end{equation*}
    with a constant $c_A:=c_A(\kappa) > 0$.
\end{lemma}
\begin{proof}
    First, recall that \cref{reg} implies that $\bar u$ has a multibang structure. Furthermore, using \cref{lemma:multi_bang} we obtain with the definition of $Q_i$ and $S_i$ in \cref{thm:optimality_set,lem:G_dir_derivative}, respectively, that $\bar u(x) \in S_i$ if and only if $\bar p(x) \in Q_i$. Now we use \cref{lem:G_dir_derivative} and the fact that $u - \bar u \in T_\Uad(\bar u)$ to compute
    \enlargethispage*{1cm}
    \begin{multline*}
        (-\bar p,u - \bar u)_{L^2(\Omega)} + \alpha G'(\bar u; u-\bar u) \\
        \begin{aligned}
            &= \int\limits_{\{\bar p \in Q_1\} } \left( -\bar p(x) + \frac{\alpha}{2} (u_1 + u_2) \right)(u(x) - \bar u(x)) \dx \\
            \MoveEqLeft[-1]+ \int\limits_{\{\bar p \in Q_d\} } \left(- \bar p(x)+ \frac{\alpha}{2} (u_{d-1} + u_d) \right)(u(x) - \bar u)(x) \dx\\
            \MoveEqLeft[-1] +\sum\limits_{i=2}^{d-1} \int\limits_{ \{\bar p \in Q_i\} \cap \{ u-\bar u \geq 0 \} } \left(-\bar p(x) + \frac{\alpha}{2}(u_i + u_{i+1}) \right) (u(x) - \bar u(x)) \dx\\
            \MoveEqLeft[-1] + \sum\limits_{i=2}^{d-1} \int\limits_{ \{\bar p \in Q_i\} \cap \{ u-\bar u < 0 \} } \left(-\bar p(x) + \frac{\alpha}{2}(u_{i-1} + u_{i}) \right) (u (x)- \bar u(x)) \dx.
        \end{aligned}
    \end{multline*}
    Here we have abbreviated the sets $\{\bar p \in Q_1\} :=\{ x \in \Omega: \bar p(x) \in Q_1 \}$ and similar for the other sets. Recall that by definition, $\bar p(x) \in Q_1$ implies that $-\bar p(x) + \frac{\alpha}{2}(u_1 + u_2) > 0$. Furthermore, we know that $\bar u(x) = u_1$, leading to $u(x) - \bar u(x) = u(x)-u_1 \geq 0$. We similarly obtain on $Q_d$ that $-\bar p(x) + \frac{\alpha}{2}(u_{d-1} + u_d) < 0$ and $u(x) - \bar u(x) = u(x) - u_d \leq 0$.
    Finally, if $\bar p(x) \in Q_i$ for $1 < i < d$, we obtain that
    \begin{equation*}
        \frac{\alpha}{2}(u_{i-1} + u_i) < \bar p(x) < \frac{\alpha}{2}(u_i + u_{i+1}),
    \end{equation*}
    which leads to
    \begin{equation*}
        - \bar p(x) + \frac{\alpha}{2}(u_i + u_{i+1}) > 0 \quad \text{and} \quad - \bar p(x) + \frac{\alpha}{2}(u_{i-1} + u_i) < 0.
    \end{equation*}
    This allows us to write
    \begin{multline*}
        (-\bar p,u - \bar u)_{L^2(\Omega)} + \alpha G'(\bar u; u-\bar u)\\
        \begin{aligned}
            &= \int\limits_{\{\bar p \in Q_1\} } \left| -\bar p(x) + \frac{\alpha}{2} (u_1 + u_2) \right||u(x) - \bar u(x)| \dx \\
            \MoveEqLeft[-1]+ \int\limits_{\{\bar p \in Q_d\} } \left|- \bar p(x)+ \frac{\alpha}{2} (u_{d-1} + u_d) \right||u(x) - \bar u(x)| \dx\\
            \MoveEqLeft[-1]+ \sum\limits_{i=2}^{d-1} \int\limits_{ \{\bar p \in Q_i\} \cap \{ u-\bar u \geq 0 \} } \left|-\bar p (x)+ \frac{\alpha}{2}(u_i + u_{i+1}) \right| |u(x) - \bar u(x)| \dx\\
            \MoveEqLeft[-1] + \sum\limits_{i=2}^{d-1} \int\limits_{ \{\bar p \in Q_i\} \cap \{ u-\bar u < 0 \} } \left|-\bar p(x) + \frac{\alpha}{2}(u_{i-1} + u_{i}) \right| |u(x) - \bar u(x)| \dx.
        \end{aligned}
    \end{multline*}
    
    Now let $\eps > 0$ and consider the set
    \begin{equation*}
        Q_1^\eps :=\left\{ q: q \leq \frac{\alpha}{2}(u_1 + u_2) - \eps \right\} \subset Q_1.
    \end{equation*}
    Let $\bar p(x) \in Q_1^\eps$. Together with $-\bar p(x) + \frac{\alpha}{2}(u_1 + u_2) > 0$, this implies that
    \begin{equation*}
        \left|-\bar p(x) + \frac{\alpha}{2}(u_1 + u_2) \right| = -\bar p(x) + \frac{\alpha}{2}(u_1 + u_2) \geq \eps,
    \end{equation*}
    leading to
    \begin{equation*}
        \begin{aligned}
            \int\limits_{\{\bar p \in Q_1\} } \left| -\bar p + \frac{\alpha}{2} (u_1 + u_2) \right||u - \bar u| \dx 
            &\geq \int\limits_{\{\bar p \in Q_1^\eps\} } \left| -\bar p + \frac{\alpha}{2} (u_1 + u_2) \right||u - \bar u| \dx\\
            &\geq \eps \int\limits_{\{\bar p \in Q_1^\eps\} } |u - \bar u| \dx.
        \end{aligned}
    \end{equation*}
    We similarly define
    \begin{equation*}
        Q_d^\eps := \left\{ q \geq \frac{\alpha}{2}(u_{d-1} + u_d) + \eps \right\},
    \end{equation*}
    leading to
    \begin{equation*}
        \int\limits_{\{\bar p \in Q_d\} } \left|- \bar p(x)+ \frac{\alpha}{2} (u_{d-1} + u_d) \right||u(x) - \bar u(x)| \dx \geq \eps \int\limits_{\{\bar p \in Q_d^\eps\} } |u(x) - \bar u(x)| \dx,
    \end{equation*}
    as well as for $1<i<d$
    \begin{equation*}
        Q_i^\eps := \left\{ \eps + \frac{\alpha}{2}(u_{i-1} + u_i) \leq q \leq \frac{\alpha}{2}(u_i + u_{i+1}) - \eps \right\} \subset Q_i.
    \end{equation*}
    The latter leads to
    \begin{align*}
        \left|-\bar p(x) + \frac{\alpha}{2}(u_i + u_{i+1})\right| &= -\bar p(x) + \frac{\alpha}{2}(u_i + u_{i+1}) \geq \eps,\\
        \left|-\bar p(x) + \frac{\alpha}{2}(u_{i-1} + u_i)\right| &= \phantom{-}\bar p(x) - \frac{\alpha}{2}(u_{i-1} + u_i) \geq \eps
    \end{align*}
    and therefore
    \begin{multline*}
        \int\limits_{ \{\bar p \in Q_i\} \cap \{ u-\bar u \geq 0 \} } \left|-\bar p(x) + \frac{\alpha}{2}(u_i + u_{i+1}) \right| |u(x) - \bar u(x)| \dx\\
        + \int\limits_{ \{\bar p \in Q_i\} \cap \{ u-\bar u < 0 \} } \left|-\bar p(x) + \frac{\alpha}{2}(u_{i-1} + u_{i}) \right| |u(x) - \bar u(x)| \dx\\
        \begin{aligned}[t]
            &\geq \eps \int\limits_{ \{\bar p \in Q_i^\eps\} \cap \{ u-\bar u \geq 0 \} } |u(x) - \bar u(x)| \dx + \eps \int\limits_{ \{\bar p \in Q_i^\eps\} \cap \{ u-\bar u < 0 \} } |u(x)-\bar u(x)| \dx\\
            &= \eps \int\limits_{ \{\bar p \in Q_i^\eps\} } |u(x)-\bar u(x)| \dx.
        \end{aligned}
    \end{multline*}

    We now combine all these estimates to obtain
    \begin{multline*}
        (-\bar p,u - \bar u)_{L^2(\Omega)} + \alpha G'(\bar u; u - \bar u) \\
        \begin{aligned}
            &\geq \eps \sum\limits_{i=1}^d \int\limits_{ \{ \bar p \in Q_i^\eps \} } |u(x) - \bar u(x)| \dx\\
            &=\eps \sum\limits_{i=1}^d \left( \int\limits_{ \{ \bar p \in Q_i \} } |u(x) - \bar u(x)| \dx - \int\limits_{ \{ \bar p \in Q_i \} \setminus \{ \bar p \in Q_i^\eps \} } |u(x) - \bar u(x)| \dx \right)\\
            &= \eps \|u - \bar u\|_{L^1(\Omega)} - \eps \sum\limits_{i=1}^d \int\limits_{ \{ \bar p \in Q_i \} \setminus \{ \bar p \in Q_i^\eps \} } |u(x) - \bar u(x)| \dx\\
            &\geq \eps \|u - \bar u\|_{L^1(\Omega)} - \eps \|u - \bar u\|_{L^\infty(\Omega)} \sum\limits_{i=1}^d \int\limits_{ \{ \bar p \in Q_i \} \setminus \{ \bar p \in Q_i^\eps \} } 1 \dx,
        \end{aligned}
    \end{multline*}
    where we have used the $L^\infty$-boundedness of $u - \bar u$ in the last step. We now use \cref{reg} to estimate the remaining sum, yielding
    \begin{align*}
        \sum\limits_{i=1}^d \int\limits_{ \{ \bar p \in Q_i \} \setminus \{ \bar p \in Q_i^\eps \} } 1 \dx = \meas\left( \bigcup\limits_{i=1}^{d-1} \left\{x\in\Omega: \left|\bar p(x) - \frac{\alpha}{2}(u_i + u_{i+1}) \right| < \eps \right\} \right) \leq c \eps^\kappa.
    \end{align*}
    Summarizing, we have for a constant $c > 1$ that
    \begin{equation*}
        (-\bar p,u - \bar u)_{L^2(\Omega)} + \alpha G'(\bar u; u - \bar u) \geq \eps \|u - \bar u\|_{L^1(\Omega)} -c \eps^{\kappa + 1},
    \end{equation*}
    and hence setting
    \begin{equation*}
        \eps := c^{- \frac{2}{\kappa}} \|u - \bar u\|_{L^1(\Omega)}^{\frac{1}{\kappa}}
    \end{equation*}
    finishes the proof.
\end{proof}

We now have everything at hand to prove approximation error estimates.
\begin{theorem}\label{thm:conv_rates}
    Let $\bar u$ be a solution of \eqref{eq:main} with corresponding state $\bar y := K\bar u$ and assume that \cref{reg} is satisfied. Furthermore, let $u_\gamma$ be the solution of \eqref{eq:main_tik} for $\gamma>0$ with corresponding state $y_\gamma := K u_\gamma$. Then there exists a constant $c> 0$ such that 
    \begin{equation*}
        \frac{1}{\gamma}\|y_\gamma - \bar y\|_Y^2 + \frac{1}{\gamma}\|u_\gamma - \bar u\|_{L^1(\Omega)}^{1 + \frac{1}{\kappa}} + \|u_\gamma - \bar u\|_{L^2(\Omega)}^2 \leq c \gamma^\kappa.
    \end{equation*}
\end{theorem}
\begin{proof}
    First note that $G$ is a convex function and hence that
    \begin{equation*}\label{eq:direction_positive}
        G'(\bar u; u_\gamma - \bar u) + G'(u_\gamma; \bar u - u_\gamma) \leq 0.
    \end{equation*}
    We thus obtain from \cref{thm:opt_variational} and \cref{lem:improved_first_order} that
    \begin{align*}
        (-\bar p,u - \bar u)_{L^2(\Omega)} + \alpha G'(\bar u; u - \bar u) \geq c_A \|u_\gamma - \bar u\|_{L^1(\Omega)}^{1 + \frac{1}{\kappa}} \quad &\forall u \in \Uad,\\
        (-p_\gamma,u - u_\gamma)_{L^2(\Omega)}+ \alpha G'(u_\gamma; u - u_\gamma) + \gamma (u_\gamma, u - u_\gamma)_{L^2(\Omega)} \geq 0 \quad &\forall u \in \Uad.
    \end{align*}
    Inserting $u=u_\gamma$ and $u=\bar{u}$ into two above inequalities, respectively, and then adding both yields
    \begin{multline*}
        (-\bar p + p_\gamma, u_\gamma - \bar u )_{L^2(\Omega)} + \alpha( G'(\bar u; u_\gamma - \bar u) + G'(u_\gamma; \bar u - u_\gamma ))
        + \gamma (u_\gamma, \bar u - u_\gamma)_{L^2(\Omega)}\\
        \geq c_A \|u_\gamma - \bar u\|_{L^1(\Omega)}^{1 + \frac{1}{\kappa}}.
    \end{multline*}
    We now use the definition of $\bar p = K^\ast(z - K \bar u)$ and $p_\gamma = K^\ast(z - K u_\gamma)$ to deduce that
    \begin{equation*}
        (-\bar p + p_\gamma, u_\gamma - \bar u )_{L^2(\Omega)} = - \| y_\gamma - \bar y\|^2_{Y}.
    \end{equation*}
    Hence, by adding $\gamma \|\bar u - u_\gamma\|^2_{L^2(\Omega)}$ to the inequality above and rearranging terms, we obtain that
    \begin{equation*}
        \begin{aligned}
            \|y_\gamma - \bar y\|_Y^2 + c_A \|u_\gamma - \bar u\|_{L^1(\Omega)}^{1 + \frac{1}{\kappa}} + \gamma \|u_\gamma - \bar u\|_{L^2(\Omega)}^2
            &\leq \alpha( G'(\bar u; u_\gamma - \bar u) + G'(u_\gamma; \bar u - u_\gamma )) \\
            \MoveEqLeft[-1]+ \gamma (\bar u, \bar u - u_\gamma)_{L^2(\Omega)}\\
            &\leq \gamma (\bar u, \bar u - u_\gamma)_{L^2(\Omega)}\\
            &\leq c \gamma \|u_\gamma - \bar u\|_{L^1(\Omega)}\\
            &\leq \frac{c_A}{2}\|u_\gamma - \bar u\|_{L^1(\Omega)}^{1 + \frac{1}{\kappa}} + c \gamma^{\kappa+1},
        \end{aligned}
    \end{equation*}
    where we have used Young's inequality in the last step. The stated inequality now follows immediately. 
\end{proof}

\section{Discretization error estimates}\label{sec:disc_error}

In practice, the exact operator $K$ is not realizable, and a discretization 
$K_h: L^2(\Omega) \to Y_h$ with finite dimensional range $Y_h$ must be employed. Denote by $u_{\gamma,h}$ the solution of the discrete problem
\begin{equation}\label{eq:var_discrete}\tag{$P_{\gamma,h}$}
    \min\limits_{u \in \Uad} \frac{1}{2}\|K_h u - z\|_Y^2 + \alpha G(u) + \frac{\gamma}{2} \|u\|_{L^2(\Omega)}^2
\end{equation}
with corresponding state $y_{\gamma,h}:=K_hu_{\gamma,h}$ and adjoint state $p_{\gamma,h}:=K_h^*(z-y_{\gamma,h})$.
If $K$ is the solution operator of an elliptic partial differential equation and $K_h$ its finite element discretization as in the next section, \eqref{eq:var_discrete} can be interpreted as a variational discretization \cite{Hinze2005,HinzeMatthes2009}.

We assume that for all $h>0$, the estimate
\begin{equation}\label{eq:discrete_operator_ass}
    \|(K -K_h) u_{\gamma,h}\|_Y + \|(K^\ast - K_h^\ast)(y_{\gamma,h} - z)\|_{L^2(\Omega)} \leq \delta(h),
\end{equation}
holds uniformly for all $\gamma>0$ with a monotonically increasing function $\delta: \R^+_0 \to \R$ such that $\delta(0) = 0$. Note that this approximation condition only needs to be satisfied for the solutions to the \emph{discretized} problem \eqref{eq:var_discrete}. However, as in \cite{Wachsmuth2013} the condition can also be replaced by a corresponding uniform condition for the solution to the continuous problem \eqref{eq:main_tik}. 

Now, we follow \cite[Proposition 1.8]{Wachsmuth2013} and estimate the discretization error for the solution to \eqref{eq:main_tik}.
\begin{theorem}\label{thm:disc_error}
    For all $\gamma >0$ and $h\geq 0$ there holds
    \begin{equation*}
        \|y_\gamma - y_{\gamma,h}\|_Y^2 + \gamma \|u_\gamma - u_{\gamma,h}\|_{L^2(\Omega)}^2 \leq (1 + \gamma^{-1})\delta(h)^2.
    \end{equation*}
\end{theorem}
\begin{proof}
    With $u_{\gamma,h}$ and $u_\gamma$ solutions to \eqref{eq:var_discrete} and \eqref{eq:main_tik}, respectively, we have from \cref{thm:opt_variational} that
    \begin{align*}
        \left( -p_{\gamma,h} + \gamma u_{\gamma,h}, u_\gamma - u_{\gamma,h} \right)_{L^2(\Omega)} + \alpha G'(u_{\gamma,h}; u_\gamma - u_{\gamma,h}) &\geq 0 ,\\
        \left( -p_\gamma + \gamma u_\gamma, u_{\gamma,h} - u_\gamma \right)_{L^2(\Omega)} + \alpha G'(u_\gamma; u_{\gamma,h} - u_\gamma) &\geq 0.
    \end{align*}
    Adding these two inequalities, substituting $p_{\gamma,h} = -K_h^*(K_hu_{\gamma,h} - z),p_\gamma=-K^*(Ku_\gamma - z)$, and using the convexity of $G$ then yields
    \begin{multline*}
        \left( K_h^*(K_h u_{\gamma,h} - z) + \gamma u_{\gamma,h}, u_\gamma - u_{\gamma,h}\right) + \left( K^*(K u_\gamma - z) + \gamma u_\gamma, u_{\gamma,h} - u_\gamma\right)\\
        \ge -\alpha \left(G'(u_{\gamma,h}; u_\gamma - u_{\gamma,h}) + G'(u_\gamma; u_{\gamma,h} - u_\gamma)\right)\ge 0.
    \end{multline*}
    We thus obtain that
    \begin{equation*}
        \begin{aligned}
            \gamma\|u_{\gamma,h}-u_\gamma\|_{L^2(\Omega)}^2 & \le\left( K_h^*(y_{\gamma,h} - z) - K^*(y_\gamma - z),u_\gamma - u_{\gamma,h}\right)\\
                                                            & \le \left((K_h^* - K^*)(y_{\gamma,h} - z),u_\gamma - u_{\gamma,h}\right) + \left(K^*(y_{\gamma,h} - y_\gamma),u_\gamma - u_{\gamma,h}\right).
        \end{aligned}		
    \end{equation*}
    The rest of the proof follows similarly to the proof of \cite[Proposition 1.6]{Wachsmuth2013}. The first term on the right-hand side is estimated by the Cauchy--Schwarz inequality and the inequality \eqref{eq:discrete_operator_ass} as
    \begin{equation*}
        \left((K_h^* - K^*)(y_{\gamma,h} - z),u_\gamma - u_{\gamma,h}\right) \le \frac{\gamma}{2}\|u_{\gamma,h}-u_\gamma\|_{L^2(\Omega)}^2 + \frac{1}{2\gamma}\delta(h)^2.
    \end{equation*}
    Rewriting the second term and using again the Cauchy--Schwarz inequality combined with the inequality \eqref{eq:discrete_operator_ass}, we obtain
    \begin{equation*}
        \begin{aligned}
            \left(K^*(y_{\gamma,h} - y_\gamma),u_\gamma - u_{\gamma,h}\right) &= -\|y_\gamma - y_{\gamma,h}\|_Y^2 + (y_\gamma - y_{\gamma,h}, (K_h - K)u_{\gamma,h})\\
                                                                              & \le -\frac{1}{2}\|y_\gamma - y_{\gamma,h}\|_Y^2 + \frac{1}{2}\delta(h)^2.
        \end{aligned}
    \end{equation*}
    Adding these two estimates, we finally arrive at
    \begin{equation*}
        \frac{1}{2}\|y_\gamma - y_{\gamma,h}\|_Y^2 + \frac{\gamma}{2} \|u_\gamma - u_{\gamma,h}\|_{L^2(\Omega)}^2 \leq \left(\frac{1}{2} + \frac{1}{2\gamma}\right)\delta(h)^2.
        \qedhere
    \end{equation*}
\end{proof}

Combining the approximation error estimate from \cref{thm:conv_rates} and the discretization error estimate from \cref{thm:disc_error}, we immediately obtain the following result.
\begin{theorem}\label{thm:discretization_error_estimate}
    If $\bar u$ satisfies \cref{reg}, then
    \begin{equation*}
        \frac{1}{\gamma}\|y_{\gamma, h} - \bar y\|_Y^2 + \|u_{\gamma, h} - \bar u\|_{L^2(\Omega)}^2 \leq c\left( \gamma^{-1}(1+\gamma^{-1}) \delta (h)^2 + \gamma^\kappa \right)
    \end{equation*}
    holds for all $\gamma>0$ and $h \geq 0$.
\end{theorem}

\section{Active set method for the regularized problem}\label{sec:active_set}

Let us now consider the special case where $y = Ku$ is given as the unique solution of the partial differential equation
\begin{equation}\label{eq:pde}
    \left\{
        \begin{aligned}
            Ay = u \quad & \text{in} \quad \Omega,\\
            y = 0 \quad &\text{on} \quad \partial \Omega.
        \end{aligned}
    \right.
\end{equation}
with $A$ being a second-order linear differential operator, e.g., $A = - \Delta$. 
In this case, the optimality conditions from \cref{thm:optimality_set_reg} can be solved using a superlinearly convergent semi-smooth Newton method in function space; see \cite{CK:2013,CK:2015,ClasonDo:2017}.

We recall that \eqref{eq:opt_reg} can be written as $u_\gamma = H_\gamma(p_\gamma)$ for $H_\gamma:L^r(\Omega)\to L^2(\Omega)$ with $r\geq 2$, 
\begin{equation*}
    [H_\gamma(p)](x) = \begin{cases}
        u_i & \text{if } p(x) \in Q_i^\gamma,\qquad 1 \leq i \leq d,\\
        \frac{1}{\gamma}\left( p(x) - \frac{\alpha}{2}(u_i + u_{i+1}) \right) & \text{if } p(x) \in Q_{i,i+1}^\gamma,\quad 1 \leq i < d,
    \end{cases}
\end{equation*}
where $p_\gamma\in H^1_0(\Omega)$ is the solution to the adjoint equation
\begin{equation}\label{eq:adjoint}
    \left\{
        \begin{aligned}
            A^*p &= z-y_\gamma \quad  \text{in} \quad \Omega,\\
            p &= 0 \quad\quad\quad \text{on} \quad \partial \Omega,
        \end{aligned}
    \right.
\end{equation}
and $y_\gamma$ is the solution to \eqref{eq:pde} with $u=u_\gamma$. From the regularity theory for \eqref{eq:adjoint}, the Sobolev embedding $H^1_0(\Omega)\hookrightarrow L^r(\Omega)$ for some $r>2$, and the general theory of semi-smooth Newton methods in function space \cite{Ulbrich:2011}, we deduce that the superposition operator $H_\gamma$ is Newton differentiable from $L^r(\Omega)$ to $L^2(\Omega)$ with
\begin{equation*}
    [D_N H_\gamma (p) h](x) = \begin{cases}
        \frac{1}{\gamma}h(x) & \text{if } p(x) \in Q_{i,i+1}^\gamma,\\
        0 & \text{else}.
    \end{cases}
\end{equation*}
A Newton step for the solution of \eqref{eq:main_tik} can therefore be formulated as
\begin{equation}\label{eq:newton_step}
    \begin{pmatrix}
        -\text{Id} & A & 0 \\
        0 & \text{Id} & A^\ast \\
        0 & A & -D_N H_\gamma(p^k) 
        \end{pmatrix} \begin{pmatrix}
        u^{k+1} - u^k\\
        y^{k+1} - y^k \\
        p^{k+1} - p^k
        \end{pmatrix} = - \begin{pmatrix}
        Ay^k - u^k\\
        A^\ast p^k + y^k - z \\
        Ay^k - H_\gamma (p^k)
    \end{pmatrix}
\end{equation}

In \cite{ClasonDo:2017}, this was reduced to a symmetric system in $(y,p)$. Here, we instead consider an equivalent primal active set formulation that has proven to be more robust for small values of $\gamma$ and $h$.
In a slight abuse of notation, we introduce
\begin{equation*}
    Q_i^k:=\left\{x\in\Omega:p^k(x)\in Q_i^\gamma\right\},\qquad 1\leq i\leq d,
\end{equation*}
and similarly for $Q_{i,i+1}^k$. 
The following algorithm is an extension of the one proposed in \cite{Stadler2009} for $G(u)=\|u\|_{L^1(\Omega)}$.
\begin{myalg}\label{myalg}
    Choose initial data $u^0, p^0$ and parameters $\alpha, \gamma$, set $k=0$ and compute the sets $Q_i^0$ for $1 \leq i \leq d$ and $Q_{i,i+1}^0$ for $1 \leq i < d$.
    \begin{enumerate}
        \item Solve for $(u^{k+1},y^{k+1}, p^{k+1}, \lambda^{k+1})$ satisfying
            \begin{subequations}\label{eq:active_set}
                \begin{equation}\label{eq:tech_1}
                    \left\{
                        \begin{aligned}
                            Ay^{k+1} -u^{k+1} &=0 ,\\
                            A^*p^{k+1} + y^{k+1}-z &= 0,\\
                            - p^{k+1}+\gamma u^{k+1} + \alpha \lambda^{k+1} &= 0,
                        \end{aligned}
                    \right.
                \end{equation}
                \begin{equation}\label{eq:tech_2}
                    \left(1- \sum\limits_{i=1}^d \1_{Q_i^k} \right) \lambda^{k+1} + \left(1-\sum\limits_{i=1}^{d-1} \1_{Q_{i,i+1}^k} \right) u^{k+1}
                    =\sum\limits_{i=1}^d \1_{Q_i^k} u_i + \frac{1}{2}\sum\limits_{i=1}^{d-1} \1_{Q_{i,i+1}^k} (u_i + u_{i+1}) ,
                \end{equation}
            \end{subequations}
        \item Compute the sets $Q_i^{k+1}$ for $1 \leq i \leq d$ and $Q_{i,i+1}^{k+1}$ for $1 \leq i < d$.
        \item If $Q_i^k = Q_i^{k+1}$ for $1 \leq i \leq d$ and $Q_{i,i+1}^k = Q_{i,i+1}^{k+1}$ for $1 \leq i < d$, then go to step 4. Otherwise set $k = k+1$ and go to step 2.
        \item STOP: $u^{k+1}$ is a solution of \eqref{eq:main_tik}.
    \end{enumerate}
\end{myalg}

The stopping criterion yields solutions of \eqref{eq:main_tik}.
\begin{lemma}
    If 
    \begin{align*}
        Q_i^k &= Q_i^{k+1} \quad 1 \leq i \leq d,\\
        Q_{i,i+1}^k &= Q_{i,i+1}^{k+1} \quad 1 \leq i < d,
    \end{align*}
    then the solution $(u^{k+1}, p^{k+1})$ computed from \eqref{eq:active_set} satisfy \eqref{eq:opt_reg}. In particular, $u^{k+1}$ is a solution to \eqref{eq:main_tik}.
\end{lemma}
\begin{proof}
    Since for fixed $Q_i^k$ and $Q_{i,i+1}^k$ the solution of \eqref{eq:active_set} is unique, we have $(u^k, y^k, p^k) = (u^{k+1}, y^{k+1}, p^{k+1})$. Inserting this into \eqref{eq:tech_2} and comparing with \eqref{eq:opt_reg} yields the claim.
\end{proof}

We now show that \cref{myalg} coincides with a semi-smooth Newton method, which implies locally superlinear convergence.
\begin{theorem}
    The active set step \eqref{eq:active_set} is equivalent to the semi-smooth Newton step \eqref{eq:newton_step}.
\end{theorem}
\begin{proof}
    Clearly, the first two equations of \eqref{eq:newton_step} are equivalent to the first two equation of \eqref{eq:tech_1}. It therefore remains to consider the last equation, which is given by     \begin{equation}\label{eq:third_row}
        A(y^{k+1} - y^k) - D_N H_\gamma (p^k) (p^{k+1} - p^k) = - Ay^k + H_\gamma (p^k).
    \end{equation}
    Let us define the function
    \begin{equation*}
        \lambda^{k+1}(x) := \begin{cases}
            -\frac{1}{\alpha} \left( - p^{k+1}(x) + \gamma u^{k+1} \right) & \text{if } x \in Q_i^k,\\
            \frac{1}{2}(u_i + u_{i+1}) & \text{if } x \in Q_{i,i+1}^k.
        \end{cases}
    \end{equation*}
    We now make a case distinction pointwise almost everywhere.
    \begin{enumerate}[label=(\roman*)]
        \item If $x\in Q_i^k$, \eqref{eq:third_row} reduces to $[Ay^{k+1}](x) = u_i$, and from the first line of \eqref{eq:newton_step} we obtain $u^{k+1}(x) = u_i$.

        \item If $x\in Q_{i,i+1}^k$, \eqref{eq:third_row} shows that 
            \begin{equation*}
                \gamma u^{k+1}(x) - p^{k+1}(x) + \frac{\alpha}{2}(u_i + u_{i+1}) = \gamma u^{k+1}(x) - p^{k+1}(x) + \alpha \lambda^{k+1}(x)=0.
            \end{equation*}
    \end{enumerate}
    Hence the third row of \eqref{eq:newton_step} is equivalent to \eqref{eq:tech_2}. In both cases, we obtain from the definition of $\lambda^{k+1}$ that
    \begin{equation*}
        -p^{k+1} + \gamma u^{k+1} + \alpha \lambda^{k+1} = 0,
    \end{equation*}
    which finally gives \eqref{eq:tech_1} and therefore the claimed equivalence. 
\end{proof}

\section{Numerical results}\label{sec:numerics}

In this section we present some numerical results and convergence rates. Let $\Omega \subset \R^d$ be a bounded Lipschitz domain and $K$ be the operator mapping $u$ to the weak solution $y$ of 
\begin{equation}\label{eq:PDE_linear}
    \left\{
        \begin{aligned}
            - \Delta y = u \quad & \text{in} \quad \Omega,\\
            y = 0 \quad &\text{on} \quad \partial \Omega.
        \end{aligned}
    \right.
\end{equation}
The operator $K_h$ is correspondingly defined via the Galerkin approximation of \eqref{eq:PDE_linear} using linear finite elements on a triangulation of $\Omega$, which is chosen in such a way that the approximation condition \eqref{eq:discrete_operator_ass} is satisfied; see \cite{Wachsmuth2013}.
For the multibang penalty, we take $(u_1,\dots,u_5) = (-2, -1, 0, 1, 2)$ and $\alpha = 2$.
We implemented \cref{myalg} in Python using DOLFIN \cite{LoggWells2010a,LoggWellsEtAl2012a}, which is part of the open-source computing platform FEniCS \cite{AlnaesBlechta2015a,LoggMardalEtAl2012a}. The linear system \eqref{eq:active_set} arising from the active set step is solved using the sparse direct solver \verb!spsolve! from SciPy. The code used to obtain the following results can be downloaded from \url{https://github.com/clason/multibangestimates}.

\paragraph[Example~1]{Example 1: $\kappa = 1$}\label{subsec:example1}

We first consider $\Omega=(0,1)$ and define
\begin{align*}
    \bar p(x) &:= 
    \begin{aligned}[t]
        &\left(\tfrac{27}{2} x\right)\1_{[0,\frac{2}{9})}(x)\\
        &+\left(-72 + \tfrac{3123 x}{2} - 13122 x^2 + 54675 x^3 - 111537 x^4 + \tfrac{177147}{2} x^5\right) \1_{[\frac{2}{9}, \frac{3}{9})}(x)\\
        &+\left(9 - 18 x\right)\1_{[\frac{3}{9}, \frac{6}{9})}(x)\\
        &+\left(-20079 + 136062 x - 367416 x^2 + 494262 x^3 - \tfrac{662661}{2} x^4 +\tfrac{177147}{2} x^5\right)\1_{[\frac{6}{9}, \frac{7}{9})}(x)\\
        &+\left(-\tfrac{27}{2} + \tfrac{27}{2} x\right)\1_{[\frac{7}{9},1]}(x),
    \end{aligned}\\
    \bar u(x) &:= 
    \1_{[\frac{2}{27}, \frac{2}{9})}(x) 
    + 2\1_{[\frac{2}{9}, \frac{3}{9})}(x) 
    + \1_{[\frac{3}{9}, \frac{4}{9})}(x)
    -\1_{[\frac{5}{9}, \frac{6}{9})}(x)
    -2\1_{[\frac{6}{9}, \frac{7}{9})}(x)
    -\1_{[\frac{7}{9}, \frac{25}{27})}(x),\\
    \bar y(x) &:= \sin(2 \pi x)\\
    e_\Omega &:= -\Delta \bar y - \bar u,\\
    z &:= - Ke_\Omega - \Delta \bar p + \bar y,
\end{align*}
see \cref{fig:adjoint1,fig:control}.
\begin{figure}[tbp]
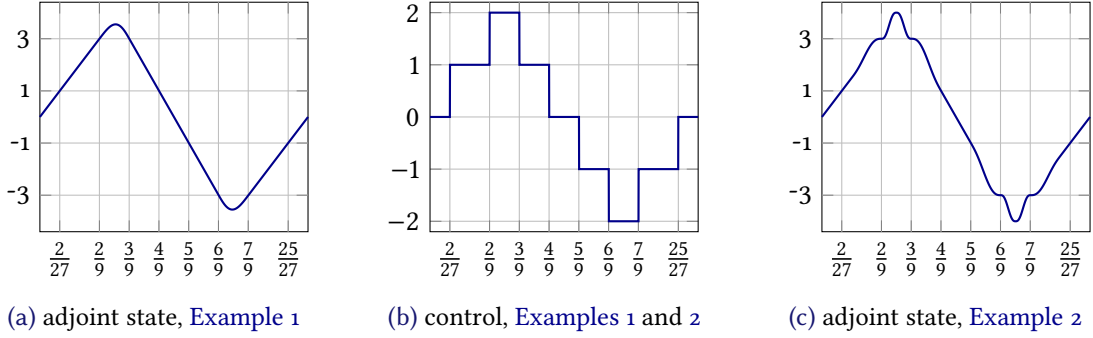

    \centering
    \begin{subfigure}[t]{0.3\textwidth}
        \centering
        \input{adjoint_example1.tikz}
        \caption{adjoint state, \nameref{subsec:example1}}
        \label{fig:adjoint1}
    \end{subfigure}
    \hfill
    \begin{subfigure}[t]{0.3\textwidth}
        \input{control_example1.tikz}
        \caption{control, \hyperref[subsec:example1]{Examples 1} and \hyperref[subsec:example2]{2}}
        \label{fig:control}
    \end{subfigure}
    \hfill
    \begin{subfigure}[t]{0.3\textwidth}
        \centering
        \input{adjoint_example2.tikz}
        \caption{adjoint state, \nameref{subsec:example2}}
        \label{fig:adjoint2}
    \end{subfigure}
    \caption{constructed optimal adjoint states $\bar p$ and optimal control $\bar u$}
    \label{fig:plot_1d}
\end{figure}
Note that $\bar p, \bar y \in C^2(\overline \Omega)$, and that $\bar u$ and $\bar p$ satisfy the optimality conditions in \cref{thm:optimality_set}. Hence, $(\bar u, \bar p)$ are a solution to \eqref{eq:main}. From \cref{thm:suffreg} we further deduce that \cref{reg} is satisfied with $\kappa = 1$.

We now compute the solution of \eqref{eq:var_discrete} for different values of $h$, where $\Omega$ is divided into equidistant elements with mesh size $h$. From \cref{thm:conv_rates} we expect that the numerical convergence rate
\begin{equation*}
    \kappa_{\gamma,h} := \frac{1}{\log(2)} \log\left( \frac{\|u_{\gamma/2,h} - \bar u\|_{L^2(\Omega)}^2}{\|u_{\gamma,h} - \bar u\|_{L^2(\Omega)}^2} \right)
\end{equation*}
satisfies $\kappa_{\gamma,h}\geq\kappa =1$. We compute $\kappa_{\gamma,h}$ for different but fixed mesh sizes $h$. Due to the discretization error, we expect a certain saturation effect for small $\gamma$; see \cref{thm:discretization_error_estimate}. Note that for $d=2$, it is known that \cref{reg} is not only sufficient for convergence rates similar to \cref{thm:conv_rates} but also necessary for high convergence rates; see \cite{WachsmuthWachsmuth2013}. Hence, we expect that $\kappa_{\gamma,h} \approx 1$, which can be observed from \cref{tab:order:ex1,fig:conv_error:ex1}. 
In addition, the discretization error dominates for small $\gamma$ as expected.
\begin{table}[t]
    \captionabove{computed numerical order of convergence for different $h$\label{tab:order}}
    \centering
    \begin{subtable}[t]{0.495\textwidth}
        \caption{\nameref{subsec:example1}}\label{tab:order:ex1}
        \centering
        \begin{tabular}[t]{
            c
            S[table-format=1.4]
            S[table-format=1.4]
            S[table-format=1.4]
        }
            \toprule
            & \multicolumn{3}{c}{$\kappa_{\gamma,h}$} \\
            \midrule
            $\gamma \setminus h$ & {$10^{-4}$} & {$10^{-5}$} & {$10^{-6}$}\\
            \midrule
            $2^{-4}$ & 1.0143 & 1.0142 & 1.0141 \\
            $2^{-6}$ & 1.0028 & 1.0008 & 1.0007 \\
            $2^{-8}$ & 1.0211 & 1.0004 & 0.9998 \\
            $2^{-10}$ & 0.9295 & 1.0038 & 0.9989 \\
            $2^{-12}$ & 0.6828 & 1.0049 & 0.9954 \\
            $2^{-14}$ & 0.0 & 0.9592 & 0.9917 \\
            $2^{-16}$ & 0.0 & -0.0096 & 0.9701 \\
            $2^{-18}$ & 0.0 & 0.0 & 0.1308 \\
            \bottomrule
        \end{tabular}
    \end{subtable}
    \hfill
    \begin{subtable}[t]{0.495\textwidth}
        \caption{\nameref{subsec:example2}}\label{tab:order:ex2}
        \centering
        \begin{tabular}[t]{
            c
            S[table-format=1.4]
            S[table-format=1.4]
            S[table-format=1.4]
        }
            \toprule
            & \multicolumn{3}{c}{$\kappa_{\gamma,h}$} \\
            \midrule
            $\gamma \setminus h$ & {$10^{-4}$} & {$10^{-5}$} & {$10^{-6}$}\\
            \midrule
            $2^{-4}$  & 0.4679   & 0.4679 & 0.4679  \\
            $2^{-6}$  & 0.3993   & 0.3992 & 0.3992  \\
            $2^{-8}$  & 0.3668   & 0.3665 & 0.3664  \\
            $2^{-10}$ & 0.3509   & 0.3518 & 0.3513  \\
            $2^{-12}$ & 0.3379   & 0.3470 & 0.3453  \\
            $2^{-14}$ & 0.3293   & 0.3496 & 0.3424  \\
            $2^{-16}$ & 0.2986   & 0.3649 & 0.3413  \\
            $2^{-18}$ & 0.1774   & 0.4122 & 0.3274  \\
            \bottomrule
        \end{tabular}
    \end{subtable}
\end{table}

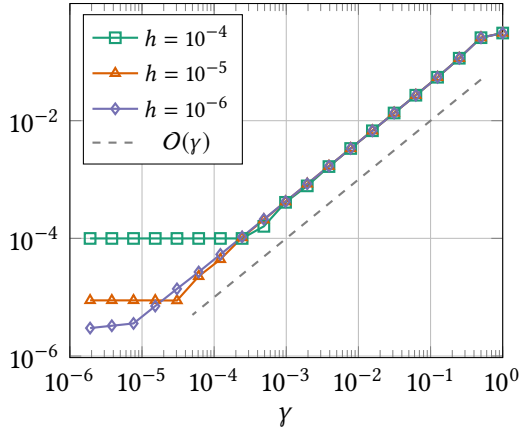
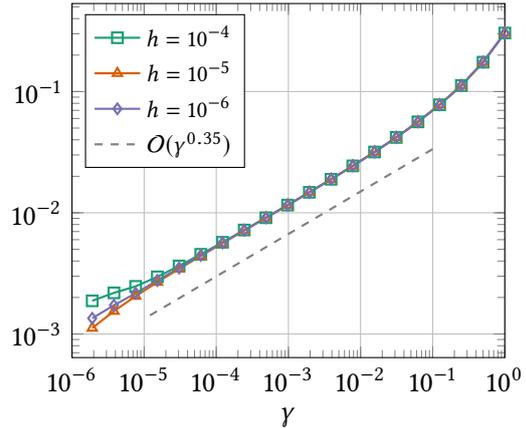
\begin{figure}[tbp]
    \centering
    \begin{subfigure}[t]{0.495\textwidth}
        \centering
        \begin{tikzpicture}
\begin{axis}[
width=\textwidth,
xlabel={$\gamma$},
ylabel style={align=center}, 
ymode=log,
xminorticks=true,
xmode=log,
xmin=1e-6,
xmax=1,
grid=major,
legend entries={$h=10^{-4}$,$h=10^{-5}$,$h=10^{-6}$,$\mathcal{O}(\gamma)$},
legend pos=north west,
]

\addplot +[mark=square,thick] 
plot coordinates {
( 1.0 , 0.308593089318 )
( 0.5 , 0.257322848636 )
( 0.25 , 0.115091603818 )
( 0.125 , 0.0546875179195 )
( 0.0625 , 0.0270736840352 )
( 0.03125 , 0.0135032125372 )
( 0.015625 , 0.00673853393574 )
( 0.0078125 , 0.00336858061005 )
( 0.00390625 , 0.00165986852541 )
( 0.001953125 , 0.000779915499445 )
( 0.0009765625 , 0.000409478787843 )
( 0.00048828125 , 0.000159980428218 )
( 0.000244140625 , 9.96632e-05 )
( 0.0001220703125 , 9.96632e-05 )
( 6.103515625e-05 , 9.96632e-05 )
( 3.0517578125e-05 , 9.96632e-05 )
( 1.52587890625e-05 , 9.96632e-05 )
( 7.62939453125e-06 , 9.96632e-05 )
( 3.81469726562e-06 , 9.96632e-05 )
( 1.90734863281e-06 , 9.96632e-05 )
};

\addplot+[mark=triangle,thick] 
plot coordinates {
( 1.0 , 0.30859321554 )
( 0.5 , 0.257323095428 )
( 0.25 , 0.115092126095 )
( 0.125 , 0.0546888348505 )
( 0.0625 , 0.0270771624453 )
( 0.03125 , 0.01350728349 )
( 0.015625 , 0.00674970061895 )
( 0.0078125 , 0.00337425488542 )
( 0.00390625 , 0.00168660664299 )
( 0.001953125 , 0.000842370250926 )
( 0.0009765625 , 0.000420063824396 )
( 0.00048828125 , 0.00020714280716 )
( 0.000244140625 , 0.000103218968084 )
( 0.0001220703125 , 4.39593043564e-05 )
( 6.103515625e-05 , 2.2610144561e-05 )
( 3.0517578125e-05 , 8.77297820536e-06 )
( 1.52587890625e-05 , 8.83199999999e-06 )
( 7.62939453125e-06 , 8.83199999999e-06 )
( 3.81469726562e-06 , 8.83199999999e-06 )
( 1.90734863281e-06 , 8.83199999999e-06 )
};
		
\addplot+[mark=diamond,thick] 
plot coordinates {
( 1.0 , 0.308593489781 )
( 0.5 , 0.257324250097 )
( 0.25 , 0.115093644562 )
( 0.125 , 0.0546897724977 )
( 0.0625 , 0.0270780350734 )
( 0.03125 , 0.0135081528524 )
( 0.015625 , 0.00675059497386 )
( 0.0078125 , 0.00337517128135 )
( 0.00390625 , 0.0016878760218 )
( 0.001953125 , 0.000844262108757 )
( 0.0009765625 , 0.000422460860162 )
( 0.00048828125 , 0.000211558601444 )
( 0.000244140625 , 0.000106115637211 )
( 0.0001220703125 , 5.33426225853e-05 )
( 6.103515625e-05 , 2.68256542029e-05 )
( 3.0517578125e-05 , 1.38530957485e-05 )
( 1.52587890625e-05 , 7.07175036881e-06 )
( 7.62939453125e-06 , 3.5817627808e-06 )
( 3.81469726562e-06 , 3.27123739998e-06 )
( 1.90734863281e-06 , 2.9868559094e-06 )
};

\addplot [domain=0.5:0.00005,gray,dashed,thick]{0.1*x};
	
\end{axis}
\end{tikzpicture}
        \caption{\nameref{subsec:example1}}\label{fig:conv_error:ex1}
    \end{subfigure}
    \hfill
    \begin{subfigure}[t]{0.495\textwidth}
        \centering
        \begin{tikzpicture}
\begin{axis}[
width=\textwidth,
xlabel={$\gamma$},
ylabel style={align=center}, 
ymode=log,
xminorticks=true,
xmode=log,
xmin=1e-6,
xmax=1,
grid=major,
legend entries={$h=10^{-4}$,$h=10^{-5}$,$h=10^{-6}$,$\mathcal{O}(\gamma^{0.35})$},
legend pos=north west,
]

\addplot+[mark=square,thick] 
plot coordinates {
( 1.0, 0.30414253724088613)
( 0.5, 0.1747615665779327)
( 0.25, 0.11217431279071897)
( 0.125, 0.07782806681481058)
( 0.0625, 0.056271974117169984)
( 0.03125, 0.041844213974210336)
( 0.015625, 0.03172723302458901)
( 0.0078125, 0.024381326111280317)
( 0.00390625, 0.0189080105467066)
( 0.001953125, 0.014756581328903739)
( 0.0009765625, 0.011570648407396668)
( 0.00048828125, 0.009106170848548811)
( 0.000244140625, 0.007204656286331864)
( 0.0001220703125, 0.005709457534309132)
( 6.103515625e-05, 0.004544259009545677)
( 3.0517578125e-05, 0.0036436428447372527)
( 1.52587890625e-05, 0.002962329182316212)
( 7.62939453125e-06, 0.002470504539726926)
( 3.814697265625e-06, 0.0021846592648420127)
( 1.9073486328125e-06, 0.0018779160080532992)
};

\addplot+[mark=triangle,thick] 
plot coordinates {
( 1.0, 0.30414192497980963)
( 0.5, 0.1747604459861054)
( 0.25, 0.11217410119884072)
( 0.125, 0.07782844399477816)
( 0.0625, 0.056272684004515905)
( 0.03125, 0.04184570641684169)
( 0.015625, 0.03173015474846991)
( 0.0078125, 0.02438595850674294)
( 0.00390625, 0.018914847904253195)
( 0.001953125, 0.014762252220840362)
( 0.0009765625, 0.011567686963589252)
( 0.00048828125, 0.009085819971257274)
( 0.000244140625, 0.007143526947570799)
( 0.0001220703125, 0.005614918138080865)
( 6.103515625e-05, 0.004406513906387492)
( 3.0517578125e-05, 0.003446360535002319)
( 1.52587890625e-05, 0.0026762121079309874)
( 7.62939453125e-06, 0.002053281686367624)
( 3.814697265625e-06, 0.0015430373578175578)
( 1.9073486328125e-06, 0.0011184700243211958)
};
		
\addplot+[mark=diamond,thick] 
plot coordinates {
( 1.0, 0.30414223256038636)
( 0.5, 0.17476144256507037)
( 0.25, 0.11217473568014599)
( 0.125, 0.07782915320602617)
( 0.0625, 0.05627355810440777)
( 0.03125, 0.041846866555097575)
( 0.015625, 0.03173176259360553)
( 0.0078125, 0.02438829252185674)
( 0.00390625, 0.01891836393318324)
( 0.001953125, 0.014767598380945235)
( 0.0009765625, 0.011575943636805629)
( 0.00048828125, 0.00909903261091397)
( 0.000244140625, 0.007162266449194472)
( 0.0001220703125, 0.005642507133642875)
( 6.103515625e-05, 0.00445029971107478)
( 3.0517578125e-05, 0.003503525389302103)
( 1.52587890625e-05, 0.002765414894951008)
( 7.62939453125e-06, 0.002166319956771129)
( 3.814697265625e-06, 0.001726532946458951)
( 1.9073486328125e-06, 0.0013468743106869819)
};
	
\addplot [domain=0.1:0.00001,gray,dashed,thick]{0.075*x^(0.35)};

\end{axis}
\end{tikzpicture}
        \caption{\nameref{subsec:example2}}\label{fig:conv_error:ex2}
    \end{subfigure}
    \caption{discretization and approximation error $\|u_{\gamma,h} - \bar u\|^2_{L^2(\Omega)}$ for different $\gamma$ and $h$}
    \label{fig:conv_error}
\end{figure}

\paragraph[Example~2]{Example 2: $\kappa < 1$}\label{subsec:example2}

We also consider an example where \cref{reg} is only satisfied with $\kappa < 1$. The idea is to violate the assumption of the sufficient condition presented in \cref{thm:suffreg}. We modify the adjoint state $\bar p$ from \nameref{subsec:example1} to
\enlargethispage*{1cm}
\begin{multline*}
    \bar p(x) := \left( \tfrac{27}{2} x   \right)\1_{[0,\tfrac{3}{27})}(x)\\
    \begin{aligned}[t]
        &+\left( 266085 x^5 - \tfrac{433593}{2} x^4 + \tfrac{135765}{2} x^3 - \tfrac{20437}{2} x^2 + \tfrac{6812}{9} x - \tfrac{1703}{81} \right)\1_{[\tfrac{3}{27}, \tfrac{2}{9})}(x)\\
        &+\left(  11334492 x^5 - 14168034 x^4 + 7054821x^3 - \tfrac{3498235}{2} x^2 +  \tfrac{1943450}{9}x - \tfrac{860051}{81}  \right)\1_{[\tfrac{2}{9}, \tfrac{5}{18})}(x)\\
        &+\left(  -11334492 x^5 + 17316666 x^4 - 10553301 x^3 + \tfrac{6413635}{2} x^2 - \tfrac{1457650}{3} x + \tfrac{528697}{18}  \right)\1_{[\tfrac{5}{18}, \tfrac{3}{9})}(x)\\
        &+\left(  -\tfrac{709317}{2} x^5 + 696195 x^4 - \tfrac{1085913}{2} x^3 + 210182 x^2 -  \tfrac{121150}{3} x + \tfrac{27761}{9}  \right)\1_{[\tfrac{3}{9}, \tfrac{4}{9})}(x)\\
        &+\left( - 18x + 9   \right)\1_{[\tfrac{4}{9},\tfrac{5}{9})}(x)\\
        &+\left(  - \tfrac{707859}{2} x^5 + \tfrac{2149821}{2}x^4 - \tfrac{2604285}{2} x^3 + \tfrac{1573075}{2}x^2 -  \tfrac{710804}{3}x + \tfrac{256331}{9}  \right)\1_{[\tfrac{5}{9},\tfrac{6}{9})}(x)\\
        &+\left( -11340324 x^5 + 39376206 x^4 - 54660123 x^3 + \tfrac{75835981}{2} x^2 -  \tfrac{39434798}{3} x + \tfrac{16396175}{9}   \right)\1_{[\tfrac{6}{9},\tfrac{13}{18})}(x)\\
        &+\left( 11340324 x^5 - 42526134 x^4 + 63759915 x^3 - \tfrac{95552197}{2} x^2 +  \tfrac{161022862}{9} x - \tfrac{433967467}{162}   \right)\1_{[\tfrac{13}{18},\tfrac{7}{9})}(x)\\
        &+\left( 265356 x^5 - \tfrac{2221101}{2} x^4 + \tfrac{3712707}{2} x^3 - 1549124 x^2 +  \tfrac{11616563}{18} x - \tfrac{17395339}{162}   \right)\1_{[\tfrac{7}{9},\tfrac{8}{9})}(x)\\
            &+\left(  \tfrac{27}{2} x -\tfrac{27}{2}   \right)\1_{[\tfrac{8}{9},1]}(x),
    \end{aligned}
\end{multline*}
see \cref{fig:adjoint2}, while the remaining functions remain unchanged.
Note that for, e.g., $\hat x :=\frac{2}{9}$, we obtain $p'(\hat x)=0$ and $p(\hat x) = 3$, which violates the assumption of \cref{thm:suffreg}.
Hence we expect that $\kappa < 1$ holds, resulting in a much slower convergence speed; see \cref{thm:conv_rates}. This is corroborated by our numerical results: We obtain $\kappa_{\gamma,h} \approx 0.35 < 1$, which can be seen in \cref{tab:order:ex2,fig:conv_error:ex2}. Due to the slower convergence speed, we do not observe a saturation effect for the chosen range of $\gamma$ and $h$.

\section{Conclusions}

For optimal control problems with a convex penalty promoting minimizers that pointwise almost everywhere take on values from a given discrete set, Moreau--Yosida approximation allows the solution by a superlinearly convergent semi-smooth Newton method. On a structural assumption on the behavior of the adjoint state near singular sets, convergence rates as the approximation parameter $\gamma\to0$ can be derived. The same assumption also yields discretization error estimates for fixed $\gamma>0$. Numerical experiments corroborate the predicted rate. 

This work can be extended in a number of directions. First, an active set condition similar to \cref{reg} was derived in \cite{PoernerWachsmuth2018} for the approximation of bang-bang control of a semilinear equation and could be adapted to the multibang control setting. Of particular interest would be the extension to problems where the control enters into the principal part of an elliptic equation as in the case of topology optimization problems \cite{CK:2015,CKK:2017}.

On the other hand, the applicability of the multibang penalty $G$ to the regularization of inverse problems was demonstrated in \cite{ClasonDo:2017}. There, a condition related to \cref{reg} was used to derive strong convergence as $\alpha\to0$, albeit without rates; and a natural question is whether the more quantitative \cref{reg} would allow obtaining such rates at least in $L^2(\Omega)$. Finally, combined regularization, approximation, and discretization estimates for the convergence $(\alpha,\gamma,h)\to 0$ would be highly useful.

\section*{Acknowledgment}

This work was funded by the German Research Foundation (DFG) under grants 
Cl 487/1-1 and Wa 3626/1-1.

\printbibliography

\end{document}